\documentclass[a4paper,12pt]{article}
\usepackage{amsmath} 
\usepackage{amssymb,bm,color,enumerate} 
\usepackage{eucal}
\usepackage{geometry}
\usepackage{graphicx}
\usepackage[amsmath,thmmarks]{ntheorem}
\theoremseparator{.}

\newtheorem{theorem}{Theorem}[section]
\newtheorem{corollary}[theorem]{Corollary}
\newtheorem{lemma}[theorem]{Lemma}
\newtheorem{proposition}[theorem]{Proposition}

\theorembodyfont{\normalfont}
\newtheorem{definition}[theorem]{Definition}
\newtheorem{remark}[theorem]{Remark}

\newtheorem{example}[theorem]{Example}
\newtheorem{conjecture}[theorem]{Conjecture}

\def\proofsymbol{\rule{0.5em}{0.5em}}

\theoremheaderfont{\normalfont\itshape\bfseries}
\theoremstyle{nonumberplain}
\theoremsymbol{\enspace\ensuremath{{\proofsymbol}}}
\newtheorem{proof}{Proof}

\theoremstyle{empty}

\numberwithin{equation}{section} 

\def\A{\text{\rm vN}(G_{\alpha,q}(H_\R))}
\def\C{{\mathbb C}}
\def\R{{\mathbb R}}
\def\N{{\mathbb N}}
\def\Z{{\mathbb Z}}

\def\d{{\rm d}}
\def\6{\, {\rm d}}
\def\i{{\rm i}}
\def\ri{{\rm i}}

\def\B{b_{\alpha,q}}
\def\G{G_{\alpha,q}}

\def\r{r}
\def\F{\mathcal{F}_{\rm fin}(H)}
\def\id{I}
\def\m{\mu_{\alpha,q,x}}

\def\qMP{{\rm MP}}
\def\P{\mathcal{P}}
\def\NC{\mathcal{NC}}
\def\Out{\rm Out}
\def\In{\rm In}
\def\PB{\mathcal{P}^B}

\def\Cr{\text{\normalfont Cr}}
\def\InS{\text{\normalfont CS}}
\def\InNB{\text{\normalfont CNB}}
\def\NB{\text{\normalfont NB}}
\def\SLNB{\text{\normalfont SLNB}}
\def\SL{\text{\normalfont SL}}
\def\Pair{\text{\normalfont Pair}}
\def\Sing{\text{\normalfont Sing}}
\renewcommand{\epsilon}{\varepsilon}
\def\Cov{\text{\normalfont Cov}}
\def\H{H}
\newcommand{\e}{{\epsilon}}

\begin{document}

\title{Fock space associated to Coxeter group of type B}

\author{Marek Bo\.zejko\footnote{Institute of Mathematics, University of Wroc\l aw,
Pl.\ Grunwaldzki 2/4, 50-384 Wroc law, Poland. Email: marek.bozejko@math.uni.wroc.pl}, %
Wiktor Ejsmont\footnote{Department of Mathematical Structure Theory (Math C), TU Graz Steyrergasse 30, 8010 Graz, Austria and Department of Mathematics and Cybernetics, 
Wroc\l aw University of Economics, ul.\ Komandorska 118/120, 53-345 Wroc\l aw, Poland. Email: wiktor.ejsmont@gmail.com} %
 and Takahiro Hasebe\footnote{Department of Mathematics, Hokkaido University, Kita 10, Nishi 8, Kita-ku, Sapporo 060-0810, Japan. Email: thasebe@math.sci.hokudai.ac.jp}}
\date{}
\maketitle

\begin{abstract}
In this article we construct a generalized Gaussian process coming from Coxeter groups of type B.  It is given by creation and annihilation operators on an $(\alpha,q)$-Fock space, which satisfy the commutation 
relation 
$$
\B(x)\B^\ast(y)-q\B^\ast(y)\B(x)=\langle x, y\rangle I+\alpha\langle \overline{x}, y \rangle q^{2N}, 
$$ 
where $x,y$ are elements of a complex Hilbert space with a self-adjoint involution $x\mapsto\bar{x}$ and $N$ is the number operator with respect to the grading on the $(\alpha,q)$-Fock space. We give an estimate of the norms of creation operators. We show that the distribution of the operators $\B(x)+\B^\ast(x)$ with
respect to the vacuum expectation becomes a generalized Gaussian distribution, in the
sense that all mixed moments can be calculated from the second moments with the help of
a combinatorial formula related with set partitions. Our generalized Gaussian distribution associates the orthogonal polynomials called the $q$-Meixner-Pollaczek polynomials, yielding the $q$-Hermite polynomials when $\alpha=0$ and free Meixner polynomials when $q=0$. 
\end{abstract}
\section{Introduction}
We present some new construction of a generalized Gaussian process related to Coxeter groups of type B.  Our construction generalizes 
Bo\.zejko and Speicher's $q$-Gaussian process \cite{BS91}  on the $q$-deformed Fock space 
$
\mathcal{F}^q(H)=(\mathbb{C}\Omega)\oplus\bigoplus_{n=1}^\infty \H^{\otimes n} $ where $\Omega$ denotes the vacuum vector and $\H$ the complexification of some real separable
Hilbert space $\H_\R$  
 on which the creation and its adjoint i.e.\ annihilation operators satisfies the $q$-commutation relation:
$$
a_q(x)a_q^\ast(y)-q a_q^\ast(y)a_q(x)=\langle x,y\rangle, \qquad x,y\in H
$$
where $q\in (-1,1)$. 
The inner product on $\mathcal{F}^q(H)$, called the $q$-deformed inner product, is the sesquilinear extension of 
\begin{align} \label{q-inner}
&\left\langle x_1\otimes\dots\otimes  x_m, y_1\otimes\dots\otimes
y_n\right\rangle_q = \delta_{m,n}\sum_{\sigma\in S(n)}q^{|\sigma|}\prod_{j=1}^n \langle x_j,y_{\sigma(j)}\rangle,
\end{align}
where  $S(n)$ is the set of all the permutations of $\{1,\dots,n\}$ and $|\sigma|:=\mbox{card}\{(i,j):i<j,
\sigma(i)>\sigma(j)\}$ is the number of inversions of  $\sigma\in
S_n$. 


The study of $q$-Gaussian  distributions has been an active field of research during the last decade. 
 A noncommutative analog of a Brownian motion (or Gaussian process, more generally) is the
family of operators, $(a_q^*(x)+a_q(x))_{x\in\H}$. When equipped with the vacuum expectation state $\langle\Omega,\cdot\, \Omega\rangle_q $ the
$q$-Gaussian algebra yields a rich non-commutative probability space. For $q=1$ (corresponding to the Bose statistics)  the $q$-deformed
 operator $ a_q^*(x)+a_q(x)$ is a natural deformation of the  standard Gaussian  random variable i.e.\  its spectral measure relative to the vacuum state satisfies 
$$\langle (a_1^*(x)+a_1(x))^n \Omega,\Omega\rangle_1=\frac{1}{\sqrt{2\pi}}\int_{\mathbb{R}}t^n e^{-\frac{t^2}{2}}\,dt
$$
when $\|x\|=1$. 
 Moreover $\{a_1^*(x)+a_1(x)\}_{x\in H}$ are commutative in the classical sense. 
The  case $q=-1$ corresponds to the Fermi statistics. It should be stressed that, for $q\neq \pm 1$, the $q$-modification of the
(anti) symmetrization operator is a strictly positive operator. Therefore, unlike the
classical Bose and Fermi cases, there are no commutation relations between the creation
operators. For $q = 0$, the $q$-Fock
space recovers the full Fock space of Voiculescu's free probability \cite{V85,V1}. For $q = 0$ the $q$-Gaussian random variables are distributed according to the semi-circle law 
$$
\langle (a_0^*(x)+a_0(x))^n \Omega,\Omega\rangle_0=\frac{1}{2\pi}\int_{-2}^2 t^n\sqrt{4-t^2}\, dt 
$$
when $\|x\|=1$. 

 The study of the noncommutative Brownian motion $(a_q^\ast(x)+a_q(x))_{x\in H}$ was initiated in \cite{BS91,BKS97,BS96}. For further generalizations of
a noncommutative Brownian motion, see \cite{BLW12,Bozejko-Yoshida06,BG02,GM02a,GM02b,Kula}. 
 In particular, this setting
gives rise to $q$-deformed versions of the  stochastic integrals \cite{Anshelevich,BKS97,Donati-Martin,Sniady}, 
with recent extensions to fourth moment convergence theorem   \cite{ASN}. It is worthwhile to mention the work of Bryc
\cite{B}, where the Laha-Lukacs property for
$q$-Gaussian processes was shown. Bryc proved
that classical processes corresponding to operators which satisfy the 
$q$-commutation relations have linear regressions and quadratic conditional variances.
After \cite{BS91}, a series of papers \cite{Biane1,B12,BS94,BE,Krolak,Speicher} appeared, which studied discrete generalizations of the $q$-commutation relations. 

Blitvi\'c \cite{B12} introduced a second-parameter refinement of the $q$-Fock space,
formulated as a $(q,t)$-Fock space $\mathcal{F}_{q,t} (\H )$. It is constructed via a direct generalization of
Bo\.zejko and Speicher's framework \cite{BS91}, yielding the $q$-Fock space when $t = 1$.  The corresponding creation and annihilation operators now satisfy the commutation
relation 
$$a_{q,t}(x)a_{q,t}(y)^*-q a_{q,t}^\ast(y) a_{q,t}(x)=\langle x,y\rangle t^N, $$ 
where $N$ is the number operator with respect to the grading on $\mathcal{F}_{q,t} (\H )$. These are
the defining relations of the Chakrabarti-Jagannathan deformed quantum oscillator
algebra, see \cite{B12} and references therein for more details. The moments of
the deformed Gaussian process $(a_{q,t}(x) +a_{q,t}^\ast(x))_{x\in H}$ are encoded by the joint statistics of crossings
and nestings in pair partitions. In particular, it is shown that the distribution of a single Gaussian operator
is orthogonalized by the $(q, t)$-Hermite polynomials. 

Another generalization of the CCR and CAR was proposed in 2013 by Bryc and Ejsmont \cite{BE}. They  define a pair of
non-commutative processes on a perturbed Fock space with three parameters. Both processes
have the same univariate distributions, and satisfy a weak form of the
polynomial martingale property. The processes give two non-equivalent Fock-space realizations of the same classical Markov process: the two-parameter bi-Poisson  processes introduced in \cite{BMW1}, and constructed in \cite{BMW2}.

The goal of this paper is to introduce an $(\alpha,q)$-Gaussian process on an $(\alpha,q)$-Fock space. Our strategy is to replace the Coxeter group of type A i.e.\ the permutation group appearing in the sum \eqref{q-inner} by the Coxeter group of type B. The Coxeter group of type B can be written as $\Z_2^n \rtimes S(n) $ and hence contains the permutation group as a subgroup. The parameter $\alpha$ corresponds to the $\Z_2$ part and the parameter $q$ corresponds to $S(n)$ part. Thus our $(\alpha,q)$-Fock space gives the $q$-Fock space when $\alpha=0$, and turns out to give the $t$-free probability space \cite{BW01,W07} when $q=0$ and $\alpha=\frac{1-t}{t}$ (where $t\in (0,\infty))$. 
It would be worth mentioning that free probabilistic considerations of type B first appeared in a paper by Biane, Goodman and Nica in \cite{BGN}. Recently, connections between type B and infinitesimal free probability were put into evidence in \cite{BS,H,FN,F12}.  

The commutation relation satisfied by the creation and annihilation operators on the $(\alpha,q)$-Fock space reads 
\begin{equation}
\B(x)\B^\ast(y)-q\B^\ast(y)\B(x)=\langle x, y\rangle I+\alpha\langle  \overline{x}, y\rangle q^{2N}, 
\end{equation}
where $N$ is the number operator. The orthogonal polynomials arising in the present framework are called $q$-Meixner-Pollaczek polynomials satisfying the recurrence relation
\begin{equation}\label{recursion0}
t P_n^{(\alpha, q)}(t) = P_{n+1}^{(\alpha, q)}(t) +[n]_q(1 + \alpha q^{n-1})P_{n-1}^{(\alpha, q)}(t), \qquad n=0,1,2,\dots
\end{equation}
where $P_{-1}^{(\alpha, q)}(t)=0,P_0^{(\alpha, q)}(t)=1$. When $\alpha=0$ then we get $q$-Hermite orthogonal polynomials and when $q=0$ then we get the orthogonal polynomials associated to a symmetric free Meixner distribution.
The moments of the Gaussian process are given by
\begin{equation}\label{mixed moment}
\begin{split}
&\langle\Omega, (\B(x_1)+\B^\ast(x_1))\cdots (\B(x_{2n})+\B^\ast(x_{2n}))\Omega\rangle_{\alpha,q}\\
&\qquad\qquad\qquad=
\sum_{\pi\in \P_{2}(2 n)} q^{\Cr(\pi)} \prod_{\substack{\{i,j\} \in \pi} }\left(\langle x_i,x_j\rangle + \alpha q^{2 \Cov(\{i,j\})}\langle x_i,\overline{x_j}\rangle\right), 
\end{split}
\end{equation}
where $\P_2(2 n)$ is the set of pair partitions, $\Cr(\pi)$ denotes the number of the crossings of $\pi$ and $\Cov(\{i,j\})$ is the number of blocks of $\pi$ which covers $\{i,j\}$. 
This formula recovers the $q$ case in \cite{BS91} when $\alpha=0$. We also give an alternative expression of \eqref{mixed moment} in terms of pair partitions of type B.

The plan of the paper is following: first we present definitions and remarks on the $(\alpha,q)$-Fock space and creation, annihilation operators. Next, natural properties of creation and annihilation operators, including norm estimates and the commutation relation, are presented in the Section 2.2. The generalized Gaussian process of type B, $\G(x)=\B(x)+\B^\ast(x), x\in H$,  is studied in the Section 3. The main theorem is placed in the Section 3. We show the recurrence relation \eqref{recursion0},  give a direct Wick formula for $\B^{\e(n)}(x_n)\cdots \B^{\e(1)}(x_1)\Omega$  and then give the mixed moment formula \eqref{mixed moment}. 

\section{Fock space and creation, annihilation operators of type B}
\subsection{Definitions}
Let $\Sigma(n)$ be the set of bijections $\sigma$ of the $2n$ points $\{\pm1,\cdots, \pm n\}$ such that 
$\sigma(-k)=-\sigma(k), k=1,\dots,n$. Equipped with the composition operation as a product, $\Sigma(n)$ becomes a group and is called a \emph{Coxeter group of type B} or a \emph{hyperoctahedral group}. 
The Coxeter group $\Sigma(n)$ is generated by $\pi_0=(1,-1), \pi_i =(i,i+1)$, $i=1,\dots,n-1$. These generators satisfy the generalized braid relations $\pi_i^2=e, 0\leq i \leq n-1$,  $(\pi_0\pi_1)^4=(\pi_i \pi_{i+1})^3=e, 1\leq i < n-1$, $(\pi_i \pi_j)^2=e$ if $|i-j|\geq2, 0\leq i,j\leq n-1$.  Note that $\{\pi_i\mid i=1,\dots,n\}$ generates the symmetric group $S(n)$. 

We express $\sigma \in \Sigma(n)$ in an irreducible form
$$
\sigma=\pi_{i_1} \cdots \pi_{i_{k}}, \qquad 0\leq i_1,\dots,i_k \leq n-1,
$$
i.e., in a form with minimal length, and in this case let 
 \begin{align}
&l_1(\sigma)= \text{The number of $\pi_0$ appearing in $\sigma$}, \\
&l_2(\sigma) =  \text{The number of $\pi_i, 1 \leq i \leq n-1$, appearing in $\sigma$}. 
\end{align}
These definitions do not depend on the way we express $\sigma$ in an irreducible form, and hence $l_1(\sigma)$ and $l_2(\sigma)$ are well defined. 

Let $H_\R$ be a separable real Hilbert space and let $H$ be its complexification with inner product $\langle\cdot,\cdot\rangle$ linear on the right component and anti-linear on the left. 
When considering elements in $H_\R$, it holds true that $\langle x,y\rangle=\langle y,x\rangle$. In order to define an action of $\pi_0$, we assume that there exists a self-adjoint involution $x\mapsto \bar{x}$ for $x\in H$. 
For example we may take the identity as an involution, or if $H$ is spanned by an orthonormal basis $(e_i)_{i \in \{\pm 1,\dots, \pm n\}}$ (or $(e_i)_{i \in \mathbb{Z}\setminus\{0\}}$), then 
we may define the involution
$$
\overline{e_{i}}=e_{-i}, \qquad i \in \{\pm1,\dots,\pm n\}. 
$$

Given a self-adjoint involution on $H$ we define an action of $\Sigma(n)$ on $H^{\otimes n}$ by 
\begin{align}
&\pi_i(x_1\otimes\cdots \otimes x_n) = x_1 \otimes \cdots \otimes x_{i-1} \otimes x_{i+1} \otimes x_{i} \otimes x_{i+2}\otimes \cdots \otimes x_{n},& n \geq2,  \label{Coxeter1} \\
&\pi_0(x_1\otimes\cdots \otimes x_n)= \overline{x_1}\otimes x_2\otimes\cdots \otimes x_n,& n \geq 1. 
\label{Coxeter2}
\end{align}

Let $\F$ be the (algebraic) full Fock space over $\H$
\begin{equation}
\F:= \bigoplus_{n=0}^\infty H^{\otimes n}
\end{equation} 
with convention that $H^{\otimes 0}=\C \Omega$ is a one-dimensional space normed space along a unit vector $\Omega$. Note that elements of $\F$ are finite linear combinations of the elements from $H^{\otimes n}, n\in \N\cup\{0\}$ and we do not take the completion. 
We equip $\F$ with the inner product  
$$
\langle x_1 \otimes \cdots \otimes x_m, y_1 \otimes \cdots \otimes y_n\rangle_{0,0}:= \delta_{m,n}\prod_{i=1}^n \langle x_i, y_i\rangle.  
$$

We will deform the inner product on $\F$. 
For $\alpha, q \in[-1,1]$ we define the type B symmetrization operator on $H^{\otimes n}$, 
 \begin{align}
&P_{\alpha,q}^{(n)}= \sum_{\sigma \in \Sigma(n)} \alpha^{l_1(\sigma)} q^{l_2(\sigma)}\, \sigma,\qquad n \geq1, \\ 
&P_{\alpha,q}^{(0)}= \id_{H^{\otimes 0}}. 
\end{align}
Note that with convention $0^0=1$, we have $P_{0,0}^{(n)}=\id_{\H^{\otimes n}}$ and note also that $P_{0,q}^{(n)}$ is the $q$-symmetrization operator \cite{BS91}. 
Moreover let 
$$
P_{\alpha,q}=\bigoplus_{n=0}^\infty P_{\alpha,q}^{(n)}
$$ be the type B symmetrization operator acting on the algebraic full Fock space. 
From Bo\.zejko and Speicher \cite[Theorem 2.1]{BS94}, the operator $P_{\alpha,q}^{(n)}$ and hence $P_{\alpha,q}$ is positive. If $|\alpha|, |q|<1$ then $P_{\alpha,q}^{(n)}$ is a strictly positive operator meaning that it is positive and $\text{Ker}(P_{\alpha,q}^{(n)})=\{0\}$. 

We deform the inner product by using the type B symmetrization operator: 
\begin{equation}
\langle x_1 \otimes \cdots \otimes x_m, y_1 \otimes \cdots \otimes y_n\rangle_{\alpha,q}:=\langle x_1 \otimes \cdots \otimes x_m, P_{\alpha,q}^{(n)}(y_1 \otimes \cdots \otimes y_n)\rangle_{0,0}, 
\end{equation}
which is a semi-inner product from the positivity of $P_{\alpha,q}$ for $\alpha,q\in[-1,1]$. We restrict the parameters to the case $\alpha,q \in(-1,1)$ so that the deformed semi-inner product is an inner product.

For $x\in H$ the free right creation and free annihilation operators $\r^\ast(x), \r(x)$ on $\F$  are defined by 
\begin{align}
&\r^\ast(x)(x_1 \otimes \cdots \otimes x_n):= x_1 \otimes \cdots \otimes x_n \otimes x, &&n\geq 1\\
&\r^\ast(x)\Omega=x, \\
&\r(x)(x_1 \otimes \cdots \otimes x_n):=\langle x, x_n\rangle\, x_1 \otimes \cdots \otimes x_{n-1},&&n\geq2, \\ 
&\r(x)x_1:= \langle x,x_1 \rangle\,\Omega, \\ 
&\r(x)\Omega =0. 
\end{align}
It then holds that $\r^\ast(x)^\ast = \r(x)$ and 
$\r^\ast: H \to \mathbb{B}(\F)$ is linear, but  $\r: H \to \mathbb{B}(\F)$ is anti-linear. 

\begin{remark}
In the literature the left creation and annihilation operators are commonly used, but we consider the right operators. 
The right creation and annihilation operators have the advantage that the embedding $\Sigma(n-1)\subset \Sigma(n)$ looks more natural (see Section \ref{Sec3}) and the commutation relation 
\begin{equation}
\r(x) (P^{(n-1)}_{\alpha,q}\otimes I) = P^{(n-1)}_{\alpha,q}\r(x)
\end{equation}
holds true on $H^{\otimes n}$ for any $x\in H$. If we consider the right action of Coxeter groups then the left creation and annihilation operators are more natural. Note that recently the combination of right and left creation and annihilation operators is of interest in a different context of free probability \cite{V2,V3}. 
\end{remark}

\begin{definition} For $\alpha,q\in(-1,1)$, the algebraic full Fock space $\F$ equipped with the inner product $\langle\cdot,\cdot \rangle_{\alpha,q}$ is called the \emph{Fock space of type B} or the \emph{$(\alpha,q)$-Fock space}. 
Let $\B^\ast(x):= \r^\ast(x)$ and $\B(x)$ be its adjoint with respect to the inner product $\langle\cdot,\cdot \rangle_{\alpha,q}$. The operators $\B^\ast(x)$ and $\B(x)$ are called \emph{creation and annihilation operators of type B} or \emph{$(\alpha,q)$-creation and annihilation operators}, respectively.  
\end{definition}
It is easy to see that $\B^\ast: H \to \mathbb{B}(\F)$ is linear and $\B: H \to \mathbb{B}(\F)$ is anti-linear. 
Since $P_{0,q}$ is the $q$-symmetrization, our $(\alpha,q)$-Fock space is the $q$-Fock space when $\alpha=0$. Thus our type B setting actually generalizes the type A setting. 


\subsection{Properties of creation and annihilation operators}\label{Sec3}
There is a natural embedding $\Sigma(n-1)=\langle \pi_0, \dots, \pi_{n-2}\rangle \subset \Sigma(n)=\langle \pi_0, \dots, \pi_{n-1}\rangle$. Corresponding to the quotient $\Sigma(n-1) \backslash\Sigma(n)$, we  decompose the operator $P^{(n)}_{\alpha,q}$. 
\begin{proposition}\label{prop1}
We have the decomposition 
\begin{equation}\label{decomposition}
P^{(n)}_{\alpha,q}=(P^{(n-1)}_{\alpha,q}\otimes I)R^{(n)}_{\alpha,q} \text{~on $H^{\otimes n}$}, \qquad n\geq 1, 
\end{equation}
where
\begin{equation}
R^{(n)}_{\alpha,q} = 1+\sum_{k=1}^{n-1}q^{k}\pi_{n-1}\cdots \pi_{n-k} + \alpha q^{n-1}\pi_{n-1} \pi_{n-2} \cdots \pi_{1}\pi_0\left(1+\sum_{k=1}^{n-1}q^{k}\pi_{1}\cdots \pi_{k}\right). 
\end{equation}
\end{proposition}
\begin{proof} 
It is known that there exist unique right coset representatives for $\Sigma(n-1) \backslash\Sigma(n)$ with minimal lengths \cite[Section 1.10, Proposition]{Hum90}. Stumbo \cite{S00} showed that these coset representatives are given by 
$
\{w(k) \mid 0 \leq k \leq 2n-1\}, 
$
where $w(k)$ is the substring composed by the first $k$ elements of $\pi_{n-1} \cdots \pi_1 \pi_0 \pi_1 \cdots \pi_{n-1}$. Hence any element $\sigma \in \Sigma(n)$ has the representation $\sigma=\sigma' w(k)$ for some unique $0\leq k \leq 2n-1$ and $\sigma'\in \Sigma(n-1)$ and moreover, this representation preserves the irreducibility (see \cite[Section 1.10, Proposition]{Hum90}) and hence $l_i(\sigma)=l_i(\sigma')+l_i(w(k))$ for $i=1,2$ . Since 
$R_{\alpha,q}^{(n)}$ has the representation 
$$
R^{(n)}_{\alpha,q} = \sum_{k=0}^{2 n-1} \alpha^{l_1(w(k))}q^{l_2(w(k))} w(k), 
$$
we have the identity \eqref{decomposition}. 
\end{proof}
The operator $R_{\alpha,q}^{(n)}$ plays a central role in this paper. Firstly we can compute the annihilation operator in terms of $R_{\alpha,q}^{(n)}$. 
\begin{proposition}\label{prop2} For $n \geq1$, we have 
\begin{equation}
\B(x)=\r(x) R^{(n)}_{\alpha,q} \text{~on $H^{\otimes n}$}. 
\end{equation}
\end{proposition}

\begin{proof}
Let $f \in H^{\otimes (n-1)},g \in H^{\otimes n}$. Then 
\begin{equation}
\begin{split}
\langle f, \B(x)g \rangle_{\alpha,q} 
&=\langle \B^\ast(x)f, g \rangle_{\alpha,q} =\langle \r^\ast(x) f, g \rangle_{\alpha,q}=\langle\r^\ast(x) f,  P^{(n)}_{\alpha,q} g \rangle_{0,0}  \\
&=\langle \r^\ast(x) f,   (P^{(n-1)}_{\alpha,q}\otimes I) R^{(n)}_{\alpha,q}g \rangle_{0,0} = \langle  f,  \r(x)  (P^{(n-1)}_{\alpha,q}\otimes I) R^{(n)}_{\alpha,q}g \rangle_{0,0}. 
\end{split}
\end{equation}
Recall that  $\r(x) (P^{(n-1)}_{\alpha,q}\otimes I) h = P^{(n-1)}_{\alpha,q}\r(x)h$ for $h\in H^{\otimes n}$ and so we get  
\begin{equation}
\begin{split}
 \langle  f,  \r(x)  (P^{(n-1)}_{\alpha,q}\otimes I) R^{(n)}_{\alpha,q}g \rangle_{0,0} 
&=  \langle  f,  P^{(n-1)}_{\alpha,q} \r(x)R^{(n)}_{\alpha,q}g \rangle_{0,0}  \\
&=  \langle  f, \r(x)R^{(n)}_{\alpha,q} g \rangle_{\alpha,q}. 
\end{split}
\end{equation}
\end{proof}
\begin{theorem}\label{thm1}
Let $N$ be the number operator, i.e.\ $N(f)=n f$ for $f\in H^{\otimes n}, n\in \N\cup\{0\}$. Then 
$$
\B(x)= \r_q(x)+ \alpha  \ell_{q}(\bar{x})q^{N-1}, \qquad x \in H, 
$$ 
where 
\begin{align}
\r_q(x)(x_1\otimes \cdots \otimes x_n)= \sum_{k=1}^n q^{n-k} \langle x, x_k \rangle\, x_1\otimes \cdots \otimes \check{x}_k \otimes \cdots \otimes x_n, \label{rq}\\
\ell_q(x)(x_1\otimes \cdots \otimes x_n)= \sum_{k=1}^n q^{k-1} \langle x,x_k\rangle\, x_1\otimes \cdots \otimes \check{x}_k \otimes \cdots \otimes x_n. \label{lq}
\end{align}
Note that $\r_0(x)=\r(x)$ is the free right annihilation operator, $\ell_0(x)$ is the free left annihilation operator and  $\ell_q(x) =\r_{1/q}(x) q^{N-1}$. 
\end{theorem}
\begin{proof}
From Propositions \ref{prop1},\ref{prop2} we have 
\begin{equation}
\begin{split}
\B(x)(x_1\otimes \cdots \otimes x_n)
&= \r(x)R_{\alpha,q}^{(n)}(x_1\otimes \cdots  \otimes x_n) = R+L,  
\end{split}
\end{equation}
where 
\begin{align}
&R = \r(x)\left(1+\sum_{k=1}^{n-1}q^{k}\pi_{n-1}\cdots \pi_{n-k}\right)(x_1\otimes \cdots  \otimes x_n), \\
&L= \alpha q^{n-1} \r(x)\pi_{n-1} \pi_{n-2} \cdots \pi_{1}\pi_0\left(1+\sum_{k=1}^{n-1}q^{k}\pi_{1}\cdots \pi_{k}\right)(x_1\otimes \cdots  \otimes x_n).
\end{align}
After some computations, we get 
$R=\r_q(x)(x_1\otimes \cdots  \otimes x_n)$ and 
$$
L=\alpha \sum_{m=1}^n q^{n+m-2} \langle x, \overline{x_m}\rangle \,x_1\otimes \cdots \otimes \check{x}_m \otimes \cdots \otimes x_n
$$
which is equal to $\alpha q^{n-1} \ell_{q}(\bar{x})$ from the self-adjointness of the map $~\bar{}~$. 
\end{proof}

\begin{proposition}\label{commutation}
For $x,y \in H$ we have the commutation relation 
\begin{equation}
\B(x)\B^\ast(y)- q \B^\ast(y)\B(x)= \langle  x,y \rangle \id + \alpha \langle x, \bar{y} \rangle\, q^{2 N}. 
\end{equation}
\end{proposition}
\begin{remark} This commutation relation is quite similar to 
$$
a_{q,t}(x)a_{q,t}^\ast(y)- q a_{q,t}^\ast(y)a_{q,t}(x)= \langle x,y \rangle t^N
$$
which appeared in \cite{B12}. 
\end{remark}
\begin{proof}
From Theorem \ref{thm1}, it holds that 
\begin{align*}
&\B(x)\B^\ast(y)(x_1\otimes \cdots \otimes x_n)
=  \sum_{k=1}^n q^{n+1-k} \langle x,x_k \rangle \, x_1\otimes \cdots \otimes \check{x}_k \otimes \cdots \otimes x_n \otimes y \\
&~~~~~~~~~~~~~~~~~~~~~~~~~~~~~~~~~~~~~~+\alpha q^n\sum_{k=1}^n q^{k-1} \langle \bar{x},x_k \rangle \, x_1\otimes \cdots \otimes \check{x}_k \otimes \cdots \otimes x_n \otimes y \\
&~~~~~~~~~~~~~~~~~~~~~~~~~~~~~~~~~~~~~~+\langle  x,y \rangle\,x_1\otimes \cdots \otimes x_n  + \alpha \langle \bar{x},y \rangle\, q^{2 n}x_1\otimes \cdots \otimes x_n, \\
&q\B^\ast(y)\B(x)(x_1\otimes \cdots \otimes x_n)
=  \sum_{k=1}^n q^{n+1-k} \langle x, x_k \rangle \, x_1\otimes \cdots \otimes \check{x}_k \otimes \cdots \otimes x_n \otimes y\\
&~~~~~~~~~~~~~~~~~~~~~~~~~~~~~~~~~~~~~~+\alpha q^n\sum_{k=1}^n q^{k-1} \langle \bar{x},x_k \rangle \, x_1\otimes \cdots \otimes \check{x}_k \otimes \cdots \otimes x_n \otimes y, 
\end{align*}
and the conclusion follows. 
\end{proof}

We will study the norm of the creation operators of type B. 
Let $[n]_q$ be the $q$-number 
$$
[n]_q:= 1+q+\cdots+q^{n-1},\qquad n \geq1
$$
 and let $[n]_q!$ be the $q$-factorial 
$$
[n]_q !:= [1]_q \cdots [n]_q,\qquad n \geq1. 
$$ 
Let $(s;q)_n$ be the $q$-Pochhammer symbol
$$
(s;q)_n:= \prod_{k=1}^n(1-s q^{k-1}),\qquad s \in \R, |q|<1, n \geq1. 
$$

\begin{lemma} \label{lem4}
For $x \in H$, we get 
\begin{align}
&\|x^{\otimes n}\|_{\alpha,q}^2 = [n]_q! (-\alpha\langle x,\bar{x}\rangle/\|x\|^2; q)_n  \|x\|^{2 n}, \qquad n \geq 1, \label{eq0001}\\
&\|R_{\alpha,q}^{(n)}\|_{0,0} \leq (1+|\alpha| |q|^{n-1})[n]_q,\qquad n \geq 1. 
\end{align}
\end{lemma}
\begin{proof} 
(1)\,\, Let $c_n:=\|x^{\otimes n}\|_{\alpha,q}^2$ for $n \geq 0$ with convention $x^{\otimes 0}=\Omega$. Then $c_0=1$ and 
\begin{align*}
c_n 
&= \langle x^{\otimes n}, P_{\alpha,q}^{(n)}x^{\otimes n}\rangle_{0,0} \\
&= \langle x^{\otimes n},  (P^{(n-1)}_{\alpha,q}\otimes I)R_{\alpha,q}^{(n)} x^{\otimes n}\rangle_{0,0} \\
&= \langle x^{\otimes n},  (P^{(n-1)}_{\alpha,q}\otimes I)\left( 1+\sum_{k=1}^{n-1}q^{k}\pi_{n-1}\cdots \pi_{n-k}\right) x^{\otimes n}\rangle_{0,0} \\
&~~~~~+  \alpha q^{n-1}\langle x^{\otimes n},  (P^{(n-1)}_{\alpha,q}\otimes I) \pi_{n-1} \pi_{n-2} \cdots \pi_{1}\pi_0\left(1+\sum_{k=1}^{n-1}q^{k}\pi_{1}\cdots \pi_{k}\right)  x^{\otimes n}\rangle_{0,0}\\
&= [n]_q \langle x^{\otimes n},  (P^{(n-1)}_{\alpha,q}\otimes I) x^{\otimes n}\rangle_{0,0} +\alpha q^{n-1}[n]_q \langle x^{\otimes n},  (P^{(n-1)}_{\alpha,q}\otimes I)(x^{\otimes (n-1)} \otimes \bar{x})\rangle_{0,0}\\
&= [n]_q (\|x\|^2 + \alpha\langle x,\bar{x}\rangle\, q^{n-1}) c_{n-1}, \qquad n \geq 1, 
\end{align*}
and so the conclusion follows by induction. 

(2)\,\, Since $\|\pi_i\|_{0,0} =1$ for all $i$, the conclusion follows easily by definition of $R_{\alpha,q}^{(n)}$. 
\end{proof}

\begin{theorem} Suppose that $x \in H, x \neq0$. 
\begin{enumerate}[\rm(1)]
\item\label{item1} If $-1 < q \leq 0$ and $\alpha \langle x,\bar{x}\rangle \geq0$, then 
\begin{equation}\label{eq01}
\|\B^\ast(x)\|_{\alpha,q}= \sqrt{\|x\|^2 + \alpha \langle x,\bar{x}\rangle}. 
\end{equation}
\item\label{item2} If $-1 < q \leq 0$ and $\alpha \langle x,\bar{x}\rangle <0$, then 
\begin{equation}\label{eq02}
\frac{\|x\|}{\sqrt{1-q}}\leq \|\B^\ast(x)\|_{\alpha,q}\leq \|x\|. 
\end{equation}
\item\label{item3} If $|\alpha| \leq q <1$, then 
$$
\|\B^\ast(x)\|_{\alpha,q}= \frac{\|x\|}{\sqrt{1-q}}. 
$$
\item\label{item4} If $0<q<\alpha  \langle x,\bar{x}\rangle/\|x\|^2$, then 
$$
\frac{\|x\|}{\sqrt{1-q}}< \|\B^\ast(x)\|_{\alpha,q}\leq \sqrt{\frac{1+|\alpha|}{1-q}}\|x\|. 
$$
\item\label{item5} Otherwise, 
$$
\frac{\|x\|}{\sqrt{1-q}}\leq \|\B^\ast(x)\|_{\alpha,q}\leq \sqrt{\frac{1+|\alpha|}{1-q}}\|x\|. 
$$
\end{enumerate}
\end{theorem}
\begin{remark}
When $\alpha=-1$ then \eqref{eq01} reminds us a \emph{Lie ball} \cite[Example 3.10]{U87}. 
\end{remark}
\begin{proof}
\eqref{item1}\,\, From Proposition \ref{commutation}, for $f \in H^{\otimes n}$ and $n \geq 0$, we get 
\begin{equation}\label{eq001}
\begin{split}
\|\B^\ast(x)f\|_{\alpha,q}^2
&=\langle \B(x)\B^\ast(x)f, f\rangle_{\alpha,q} \\
&= q \langle \B^\ast(x)\B(x)f, f\rangle_{\alpha,q} + (\|x\|^2 + \alpha \langle x,\bar{x}\rangle\, q^{2n} ) \|f\|_{\alpha,q}^2\\
&\leq (\|x\|^2 + \alpha \langle x,\bar{x}\rangle\,q^{2n}) \|f\|_{\alpha,q}^2. 
\end{split}
\end{equation}
If $\alpha \langle x,\bar{x}\rangle\geq0$, then $\|\B^\ast(x)f\|_{\alpha,q}^2 \leq  (\|x\|^2 + \alpha \langle x,\bar{x}\rangle) \|f\|_{\alpha,q}^2$. The equality is achieved when $f=\Omega$, and hence \eqref{eq01} holds.  

\eqref{item2} Upper bound. If $\alpha \langle x,\bar{x}\rangle<0$, then from \eqref{eq001} one may obtain the inequality $\|\B^\ast(x)f\|_{\alpha,q}^2 \leq  \|x\|^2 \|f\|_{\alpha,q}^2$.

\eqref{item2},\eqref{item3},\eqref{item5} Lower bound. It follows from \eqref{eq0001} that 
\begin{equation}\label{eq0000}
\begin{split}
\|\B^\ast(x)x^{\otimes (n-1)}\|_{\alpha,q}^2
&=\|x^{\otimes n}\|_{\alpha,q}^2 \\
&= [n]_q!(\|x\|^2 +\alpha \langle x,\bar{x}\rangle\, q^{n-1}) \|x^{\otimes (n-1)}\|_{\alpha,q}^2. 
\end{split}
\end{equation}
By letting $n\to\infty$, the lower bound follows.

\eqref{item4},\eqref{item5} Upper bound. The proof follows the line of \cite[Lemma 4]{BS91}. We have 
\begin{equation*}
\begin{split}
(P_{\alpha,q}^{(n)})^2 
&=   P_{\alpha,q}^{(n)} (P_{\alpha,q}^{(n)})^\ast \\
&= (P_{\alpha,q}^{(n-1)}\otimes I) R_{\alpha,q}^{(n)} (R_{\alpha,q}^{(n)})^\ast  (P_{\alpha,q}^{(n-1)} \otimes I)^\ast   \\
&\leq  \|R_{\alpha,q}^{(n)}\|_{0,0}^2 (P_{\alpha,q}^{(n-1)}\otimes I) (P_{\alpha,q}^{(n-1)}\otimes I)^\ast \\
&= \|R_{\alpha,q}^{(n)}\|_{0,0}^2 ((P_{\alpha,q}^{(n-1)})^2\otimes I).  
\end{split}
\end{equation*}
By taking the square root of operators using Lemma \ref{lem4} one gets
\begin{equation}\label{eq0002}
\begin{split}
P_{\alpha,q}^{(n)} &\leq  \|R_{\alpha,q}^{(n)}\|_{0,0} (P_{\alpha,q}^{(n-1)}\otimes I)  \\
&\leq (1+|\alpha\|q|^{n-1})[n]_q (P_{\alpha,q}^{(n-1)}\otimes I) \\
&\leq \frac{1+|\alpha|}{1-q}P_{\alpha,q}^{(n-1)}\otimes I
\end{split}
\end{equation}
on $H^{\otimes n}$ regarding the inner product $\langle \cdot,\cdot\rangle_{0,0}$. Therefore we have for $f\in H^{\otimes n}$ that 
\begin{equation}\label{eq0003}
\begin{split}
\langle \B^\ast(x) f,  \B^\ast(x) f \rangle_{\alpha,q} 
&= \langle f\otimes x,  P_{\alpha,q}^{(n+1)}(f\otimes x) \rangle_{0,0} \\
&\leq \frac{1+|\alpha|}{1-q} \langle f\otimes x,  (P_{\alpha,q}^{(n)}f)\otimes x \rangle_{0,0} \\
&= \frac{1+|\alpha|}{1-q} \langle f,  P_{\alpha,q}^{(n)}f\rangle_{0,0} \langle x, x\rangle \\
&=\frac{1+|\alpha|}{1-q} \|f\|_{\alpha,q}^2 \|x\|^2,  
\end{split}
\end{equation}
and hence the inequality $\|\B^\ast(x)\|_{\alpha,q} \leq \sqrt{(1+|\alpha|)/(1-q)}\|x\|$ holds.

\eqref{item3} Upper bound. We try to refine the estimate \eqref{eq0002} by carefully looking at $(1+|\alpha\|q|^{n-1})[n]_q $ as a function of $n \geq 1$. Let
\begin{equation}\label{function}
f_n(q):= (1-q^n)(1+ b q^{n-1})=1+q^{n-1}h_n(q),  
\end{equation}
where  $b=|\alpha|$ and $h_n(q)=b - q -b q^n$. Then $h_n(b)=-b^{n+1}\leq0$, $h_n'(q)=-1 -n b q^{n-1}<0$, and hence for  $q\in [b,1)$ one obtains $h_n(q)\leq 0$ and hence $(1+|\alpha| q^{n-1})[n]_q = f_n(q)/(1-q) \leq 1/(1-q)$. 
Following the argument \eqref{eq0003}, we have the upper bound $\|\B^\ast(x)\|_{\alpha,q} \leq (1/\sqrt{1-q} )\|x\|.$ The lower bound was already obtained and so \ref{item3} follows. 

\eqref{item4} Lower bound. We will carefully study the coefficient $[n]_q!(\|x\|^2 +\alpha \langle x,\bar{x}\rangle\, q^{n-1}) = f_n(q) \|x\|^2/(1-q)$ in \eqref{eq0000}, now with $b= \alpha \langle x, \bar{x}\rangle/\|x\|^2>0$. Then $h_n'(q)<0$ and hence $h_n$ is decreasing. Moreover $h_n(b)=-b^{n+1}<0$ and $h_n(b -b^{n+1})= b^{n+1}-b (b-b^{n+1})^n > b^{n+1} - b\cdot b^n =0$. Hence $h_n$ has a unique zero  $q_n(b)$ in $(b-b^{n+1},b)$. In particular $q_n(b)\to b$ as $n \to \infty$. Therefore, for fixed $0<q < b$, the zero $q_n(b)$ is larger than $q$ for large $n$ and hence $h_n(q)$ is strictly positive. This implies that $\|\B^\ast(x)x^{\otimes (n-1)}\|_{\alpha,q}  >(1/\sqrt{1-q} )\|x^{\otimes (n-1)}\|_{\alpha,q}.$
\end{proof}

\section{Gaussian operator of type B}
\subsection{Probability density function and orthogonal polynomials}
\begin{definition}
The operator 
\begin{equation}
\G(x)= \B(x) +\B^\ast(x),\qquad x \in H
\end{equation}
on $\F$ is called the \emph{Gaussian operator of type B} or \emph{$(\alpha,q)$-Gaussian operator}. The family $\{\G(x)\mid x \in H\}$ is called the Gaussian process of type B or $(\alpha,q)$-Gaussian process. 
\end{definition}

For a probability measure $\mu$ with finite moments of all orders, let us orthogonalize the sequence $(1,t,t^2,t^3,\dots)$ in the Hilbert space $L^2(\R,\mu)$, following the Gram-Schmidt method. This procedure yields orthogonal polynomials $(P_0(t), P_1(t), P_2(t), \dots)$ with $\text{deg}\, P_n(t) =n$. Multiplying by constants, we take $P_n(t)$ to be monic, i.e., the coefficient of $t^n$ is 1. It is known that they satisfy a recurrence relation
\[
t P_n(t) = P_{n+1}(t) +\beta_n P_n(t) + \gamma_{n-1} P_{n-1}(t),\qquad n =0,1,2,\dots
\]
with the convention that $P_{-1}(t)=0$. The coefficients $\beta_n$ and $\gamma_n$ are called \emph{Jacobi parameters} and they satisfy $\beta_n \in \R$ and $\gamma_n \geq 0$. 
It is known that 
\begin{equation}\label{eq54}
\gamma_0 \cdots \gamma_n=\int_{\R}|P_{n+1}(t)|^2\mu(d t),\qquad n \geq 0.
\end{equation} 
Moreover, the measure $\mu$ has a finite support of cardinality $N$ if and only if $\gamma_{N-1}=0$ and $\gamma_n > 0$ for $n = 0,\dots, N-2$. 

The continued fraction representation of the Cauchy transform can be expressed in terms of the Jacobi Parameters:
\[
\int_{\R}\frac{\mu(d t)}{z-t} = \dfrac{1}{z-\beta_0 -\dfrac{\gamma_0}{z-\beta_1-\dfrac{\gamma_1}{z- \beta_2 - \cdots}}}.
\]
This representation is useful to calculate the Cauchy transform when Jacobi parameters are given. More details are found in \cite{HO07}. 

For $-1 < \alpha,q <1$ let $(P_n^{(\alpha,q)}(t))_{n=0}^\infty$ be the orthogonal polynomials with the recursion relation
\begin{equation}\label{recursion}
t P_n^{(\alpha, q)}(t) = P_{n+1}^{(\alpha, q)}(t) +[n]_q(1 + \alpha q^{n-1})P_{n-1}^{(\alpha, q)}(t), \qquad n=0,1,2,\dots
\end{equation}
where $P_{-1}^{(\alpha,q)}(t)=0,P_0^{(\alpha,q)}(t)=1$. These polynomials are called \emph{$q$-Meixner-Pollaczek polynomials}. The orthogonalizing probability measure $\qMP_{\alpha,q}$ is known in \cite[(14.9.4)]{KLS10}, supported on $(-2/\sqrt{1-q}, 2/\sqrt{1-q})$ and absolutely continuous with respect to the Lebesgue measure with density 
\begin{equation}\label{eq10}
\frac{d \qMP_{\alpha,q}}{dt}(t)=  \frac{(q;q)_\infty(\beta^2; q)_\infty}{2\pi\sqrt{4/(1-q) -t^2}}\cdot\frac{g(t,1;q) g(t,-1;q) g(t,\sqrt{q};q) g(t,-\sqrt{q};q)}{g(t, \i \beta;q)g(t,-\i \beta;q)}
\end{equation}
where 
\begin{align}
&g(t,b;q)= \prod_{k=0}^\infty(1-4 b t (1-q)^{-1/2} q^k + b^2 q^{2k}), \\
&(s;q)_\infty=\lim_{n\to\infty}(s; q)_n=\prod_{k=0}^\infty(1-s q^{k}),\qquad s \in \R, \\
&\beta
= 
\begin{cases}
\sqrt{-\alpha}, & \alpha \leq 0, \\
\ri \sqrt{\alpha}, & \alpha \geq0.   
\end{cases}
\end{align}
\begin{remark}
For $\alpha <0$ the recursion \eqref{recursion} coincides with \cite[(14.9.4)]{KLS10} by considering the dilation 
$P_n^{(\alpha,q)}(t)\mapsto \lambda^{-n} P_n^{(\alpha,q)}(\lambda t)$ with $\lambda = \sqrt{1-q}/2$. Hence $P_n^{(\alpha,q)}(t)$ are orthogonal polynomials associated to $\qMP_{\alpha,q}$ for $\alpha <0$ (note that (14.9.2) contains an error; the integral should be performed over the interval $[-\pi/2 -\phi, \pi/2-\phi]$, not over $[-\pi,\pi]$). 
The case $\alpha \geq0$ (or $\beta^2 \leq0$) is not covered in \cite[(14.9.2)]{KLS10}, but we can resort to the analytic continuation. 
Let $(m_n(\qMP_{\alpha,q}))_{n=0}^\infty$ be the moments of $\qMP_{\alpha,q}$ and let $h_n(s):= (2/\sqrt{1-q})^{n+1} \frac{d \qMP_{\alpha,q}}{dt}(2s/\sqrt{1-q})$. 
We know that 
\begin{equation}\label{eq11}
m_n(\qMP_{\alpha,q})=\int_{-1}^{1}\, s^n \,h_n(s) \,ds,\qquad \beta^2 >0, ~q\in(-1,1),~n=0,1,2,\dots 
\end{equation}
The moments $m_n(\qMP_{\alpha,q})$ are polynomials of the Jacobi parameters $1-\beta^2 q^{n-1}$ (see Accardi-Bo\.zejko's formula \cite[Corollary 5.1]{AB98}) and so the moments extend to entire analytic functions of $\beta^2, q$. On the other hand, $g(t, \i \beta;q)g(t,-\i \beta;q)$ is analytic in $(\beta^2,q)$ and does not have a zero in an open neighborhood $U \subset \C^2$ of $[-1,1] \times (-1,1)$, and then the function $d\qMP_{\alpha,q}/dt$ and hence (by dominated convergence) the integral $\int_{-1}^{1}\, s^n \,h_n(s) \,ds$ has analytic continuation to the same region $U$. By analytic continuation, formula \eqref{eq11} holds for $(\beta^2,q)\in U$, and hence the measure $\qMP_{\alpha,q}$ is a probability measure giving the moments $(m_n(\qMP_{\alpha,q}))_{n=0}^\infty$ for the case $\alpha \geq0$ too. 
\end{remark}

\begin{theorem} Suppose $\alpha,q\in(-1,1)$ and $x \in H, \|x\|=1$. Let $\m$ be the probability distribution of $\G(x)$ regarding the vacuum state. Then $\m$ is equal to $\qMP_{\alpha\langle x,\bar{x}\rangle,q}$. 
\end{theorem}
\begin{proof} 
Let $\gamma_{n-1}= [n]_q(1 + \alpha \langle x, \bar{x}\rangle\, q^{n-1})$. 
From Lemma \ref{lem4} we get $\|x^{\otimes n}\|_{\alpha,q}^2=[n]_q!(-\alpha\langle x,\bar{x}\rangle;q)_n=\gamma_0\gamma_1\cdots \gamma_{n-1}$ and hence from \eqref{eq54} it follows that 
\begin{equation}
\|x^{\otimes n}\|_{\alpha,q}= \|P_n^{(\alpha\langle x,\bar{x}\rangle,q)}\|_{L^2},\qquad n\in\N\cup\{0\}. 
\end{equation}
Therefore the map $\Phi\colon (\text{span}\{x^{\otimes n}\mid n \geq 0\}, \|\cdot\|_{\alpha,q}) \to L^2(\R,\qMP_{\alpha\langle x,\bar{x}\rangle,q})$ defined by $\Phi(x^{\otimes n})= P_n^{(\alpha\langle x,\bar{x}\rangle,q)}(t)$ is an isometry. 
Note that  
\begin{align*}
\G(x)x^{\otimes n} 
&= \B^\ast(x)x^{\otimes n}+\B(x)x^{\otimes n}  \\
&= x^{\otimes (n+1)}+\r(x) R_{\alpha,q}^{(n)}x^{\otimes n}  \\
&=x^{\otimes (n+1)}+[n]_q x^{\otimes (n-1)} + \alpha q^{n-1} [n]_q \langle x,\bar{x}\rangle\,x^{\otimes (n-1)}\\
&= x^{\otimes (n+1)}+[n]_q (1 + \alpha\langle x,\bar{x}\rangle q^{n-1})x^{\otimes (n-1)}, 
\end{align*}
where Proposition \ref{prop2} was used on the second line. Hence inductively we can compute $\G(x)^n\Omega$ and show that $\Phi(\G(x)^n\Omega) = t^n$. Since $\Phi$ is an isometry we get 
$\langle \Omega, \G(x)^n\Omega\rangle_{\alpha,q} = m_n(\qMP_{\alpha\langle x,\bar{x}\rangle,q})$ for even integers $n$. For odd integers $n$ we can show that $\langle \Omega, \G(x)^n\Omega\rangle_{\alpha,q} =0= m_n(\qMP_{\alpha\langle x,\bar{x}\rangle,q})$.  
Since $\qMP_{\alpha\langle x,\bar{x}\rangle,q}$ is compactly supported, probability measures giving the moment sequence $m_n(\qMP_{\alpha\langle x,\bar{x}\rangle,q})$ are unique and hence $\qMP_{\alpha\langle x,\bar{x}\rangle,q}=\m$.  
\end{proof}

The probability density functions of $\qMP_{\alpha,q}$ are shown in Fig.\ \ref{dia1}--\ref{dia6}. 
\begin{figure}[htpb]
\begin{minipage}{0.5\hsize}
\begin{center}
 \includegraphics[width=70mm]{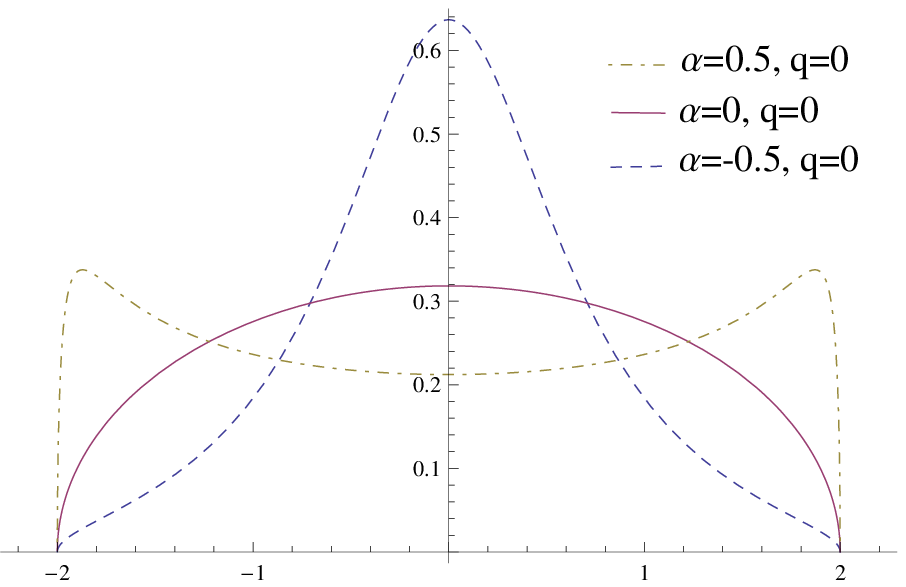}
\caption{$q=0$}\label{dia1}
\end{center}
  \end{minipage}
  \begin{minipage}{0.5\hsize}
\begin{center}
 \includegraphics[width=70mm]{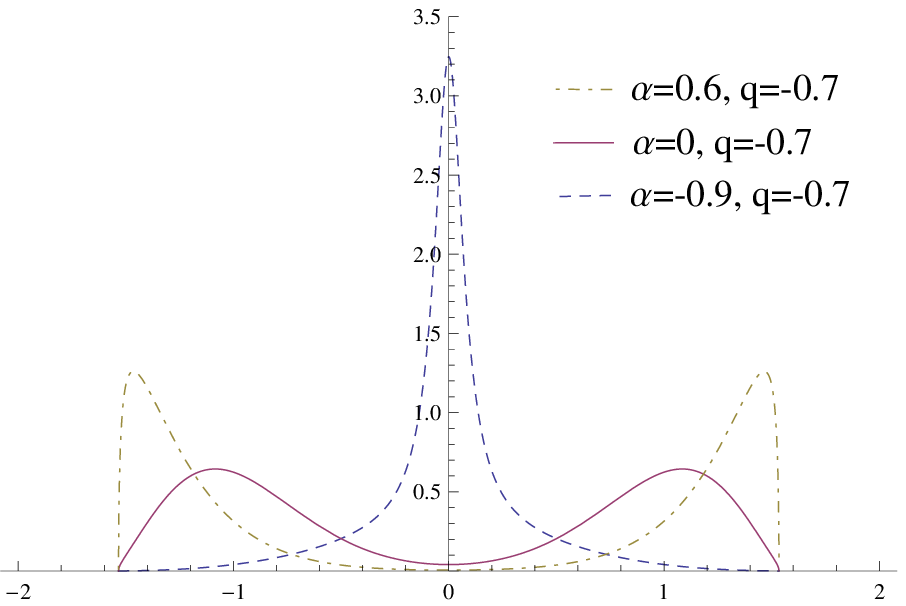}
\caption{$q=-0.7$}\label{dia2}
\end{center}
\end{minipage}
\vspace{10pt}

\begin{minipage}{0.5\hsize}
\begin{center}
 \includegraphics[width=70mm]{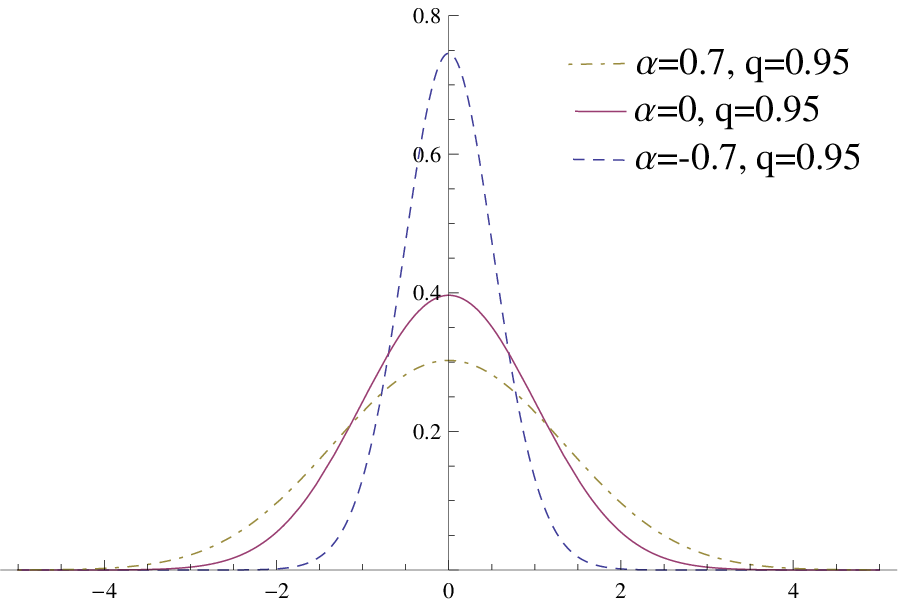}
\caption{$q=0.95$}\label{dia3}
\end{center}
\end{minipage}
\begin{minipage}{0.5\hsize}
\begin{center}
 \includegraphics[width=70mm,clip]{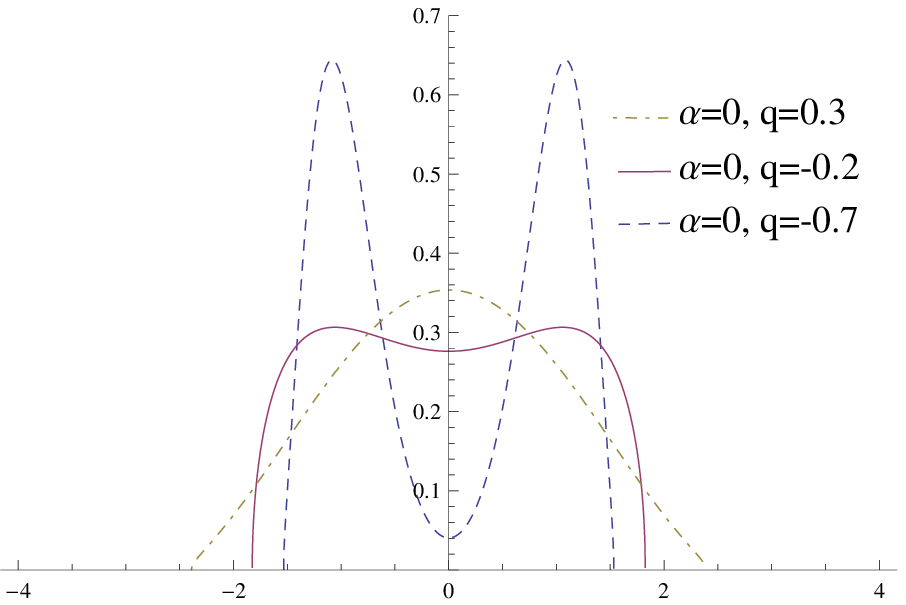}
\caption{$\alpha=0$}\label{dia4}
\end{center}
\end{minipage}

\vspace{10pt}

\begin{minipage}{0.5\hsize}
\begin{center}
 \includegraphics[width=70mm,clip]{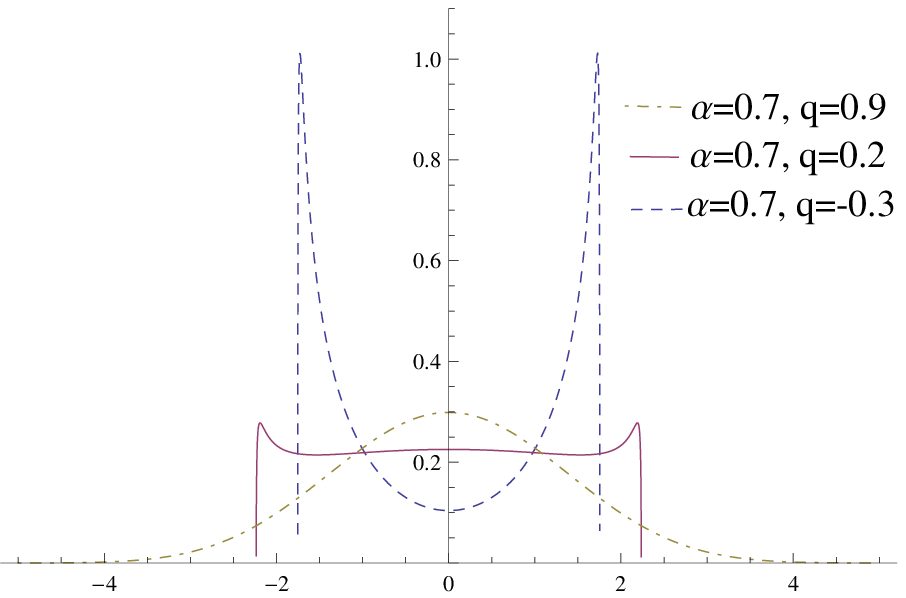}
\caption{$\alpha=0.7$}\label{dia5}
\end{center}
\end{minipage}
\begin{minipage}{0.5\hsize}
\begin{center}
 \includegraphics[width=70mm,clip]{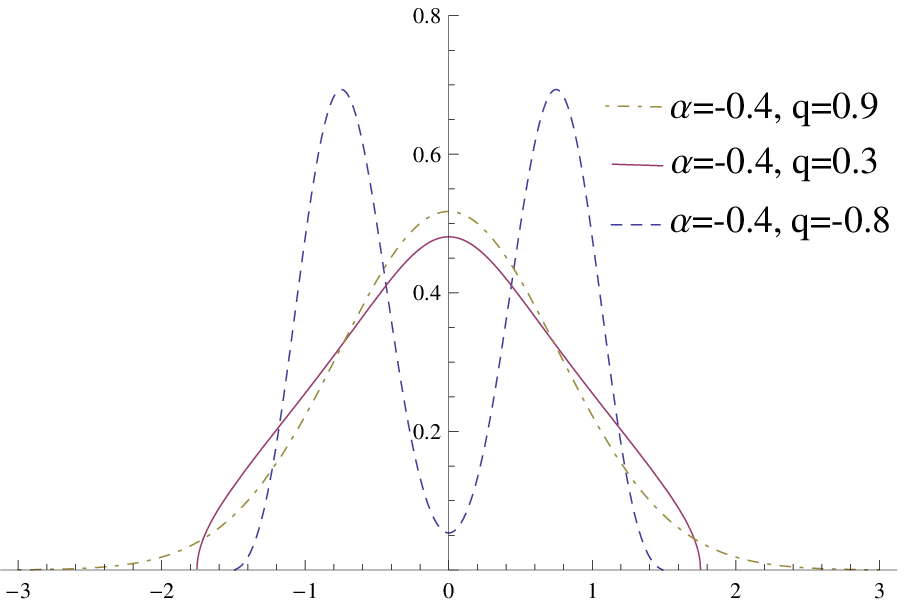}
\caption{$\alpha=-0.4$}\label{dia6}
\end{center}
\end{minipage}

\end{figure}

By weak continuity we may allow the parameters $(\alpha,q)$ of $\qMP_{\alpha,q}$ to take any values in $[-1,1] \times [-1,1]$. 
\begin{example} 
\begin{enumerate}[\rm(1)]
\item The measure $\qMP_{\alpha,1}$ is the normal law $(2(1+\alpha)\pi)^{-1/2}e^{-\frac{t^2}{2(1+\alpha)}}1_\R(t)\,\d t$. The orthogonal polynomials $P_n^{(0,1)}(t)$ are Hermite polynomials. 

\item The measure $\qMP_{0,0}$ is the standard Wigner's semicircle law $(1/2\pi)\sqrt{4-t^2}1_{(-2,2)}(t)\,\d t$.  The orthogonal polynomials $P_n^{(0,0)}(t)$ are Chebyshev polynomials of the second kind. 

\item The measure $\qMP_{0,q}$ is the $q$-Gaussian law and the orthogonal polynomials $P_n^{(0,q)}(t)$ are called $q$-Hermite polynomials. 

\item The measure $\qMP_{\alpha,-1}$ is the Bernoulli law 
$(1/2)(\delta_{\sqrt{1+\alpha}}+\delta_{-\sqrt{1+\alpha}})$. 

\item The measure $\qMP_{\alpha,0}$ is a symmetric free Meixner law~\cite{A03,BB06,SY01}. 

\item The measure $\qMP_{-1,1}$ is the delta measure $\delta_0$ and so is trivial, but we can find a nontrivial scaling. If we take  $\alpha = -q^{2\gamma}$ for $  \gamma \in (0,\infty)$ then \eqref{recursion} becomes 
\begin{equation}
t Q_n^{(\gamma,q)}(t) = Q_{n+1}^{(\gamma,q)}(t) + \frac{1}{4}[n]_q [n+2  \gamma-1]_q Q_{n-1}^{(\gamma,q)}(t), 
\end{equation}
where 
\begin{equation}
Q_n^{(\gamma,q)}(t)= \frac{P_n^{(\alpha,q)}(\lambda t)}{\lambda^n},\qquad \lambda =2\sqrt{1-q}. 
\end{equation}
In the limit $q\uparrow1$ the polynomial $Q_n^{(\gamma,q)}(t)$ tends to a polynomial which we denote by $Q_n^{(\gamma)}(t)$  and we get the recursion 
\begin{equation}
t Q_n^{(\gamma)}(t) = Q_{n+1}^{(\gamma)}(t) + \frac{1}{4}n(n+2  \gamma-1)Q_{n-1}^{(\gamma)}(t), 
\end{equation}
which gives Meixner-Pollaczek polynomials corresponding to the symmetric Meixner distribution 
\begin{equation}\label{Classical Meixner}
 \frac{4^\gamma}{2\pi \Gamma(2\gamma)}|\Gamma(  \gamma+\ri t)|^2\,1_\R(t)\, dt. 
\end{equation}
See \cite{KLS10} for further information on Meixner-Pollaczek polynomials. 
\end{enumerate}
\end{example}

Free infinite divisibility of probability measures has been studied by several authors. Of particular interests were probability measures arising from generalized Brownian motions. The moments of typical generalized Brownian motions have some combinatorial interpretations in terms of set partitions \cite{BBLS11, BW01}. Our probability measure $\qMP_{\alpha,q}$ has a combinatorial interpretation too (see Theorem \ref{thm2}) and it extends several known measures. The $q$-Gaussian distribution $\qMP_{0,q}$ is  freely infinitely divisible if and only if $q \in[0,1]$; see Anshelevich et al.\ \cite{ABBL10}. Note that the case $\alpha=0, q=1$ corresponds to the classical Gaussian distribution and its free infinite divisibility was first  proved by Belinschi et al.\ \cite{BBLS11}. The free Meixner distribution $\qMP_{\alpha,0}$ is freely infinitely divisible if and only if $\alpha \in [-1,0]$; see Saitoh and Yoshida \cite{SY01}.  
The measure $\qMP_{-q^{2  \gamma},q}$ with suitable scaling converges to the classical Meixner distribution \eqref{Classical Meixner} as $q \uparrow1$ and the limit distribution is freely infinitely divisible for $  \gamma\leq 1/2$ (the case $  \gamma >1/2$ is still open); see \cite{BH13}. 
Moreover, numerical calculations of free cumulants (thanks to Franz Lehner) support the following conjecture. 
\begin{conjecture}The probability measure $\qMP_{\alpha,q}$ is freely infinitely divisible if $\alpha \in[-1,0]$ and $q \in[0,1]$. 
\end{conjecture}

\subsection{Pair partitions of type B and moments of Gaussian operator}


Let $[n]$ be the set $\{1,\dots,n\}$. 
A \emph{pair} (or  a pair block) $V$ of a set partition is a block with cardinality 2 and a \emph{singleton} of a set partition is a block with cardinality 1. 

When $n$ is even, a set partition of $[n]$ is called a \emph{pair partition} if every block is a pair. The set of pair partitions of $[n]$ is denoted by $\P_2(n)$. 

A pair $(\pi,f)$ is called a \emph{set partition of $[n]$ of type B} if $\pi$ is a set partition of $[n]$ and $f:\pi\to \{\pm1\}$ is a coloring of the blocks of $\pi$. We denote by $\PB(n)$ the set of all set partitions of $[n]$ of type B. 
We denote by $\PB_{1,2}(n)$ the set of set partitions of $[n]$ of type B such that 
\begin{enumerate}[\rm(1)]
\item each block is a singleton or a pair; 
\item each singleton is necessarily colored by $1$. 
\end{enumerate}
When $n$ is even, we call $(\pi,f) \in \PB(n)$ a \emph{pair partition of $[n]$ of type B} if $\pi$ is a pair partition. The set of pair partitions of $[n]$ of type B is denoted by $\PB_2(n)$.

\begin{remark}
Our definition of set partitions of type B is different from \cite{R97}. 
\end{remark}

Given $\epsilon=(\epsilon(1), \dots, \epsilon(n))\in\{1,\ast\}^n$, let $\P_{1,2;\epsilon}(n)$ be the set of partitions $\pi\in \P_{1,2}(n)$ such that when $\pi$ is written as
$$
\pi =\{\{a_1,b_1\}, \dots, \{a_k,b_k\}, \{c_1\}, \dots, \{c_m\}\},\qquad k,m \in \N \cup\{0\}, a_i<b_i, i\in[k], 
$$
then $\epsilon(a_i)= \ast$ and $\epsilon(b_i)=1$ for all $1 \leq i \leq k$ and $\epsilon(c_i)=\ast$ for all $1 \leq i \leq m$. We also let $\P_{2;\epsilon}(n):=\P_{1,2;\epsilon}(n)\cap \P_2(n)$. Let $\PB_{1,2;\epsilon}(n)$ be the subset of $\PB_{1,2}(n)$ defined by $(\pi,f)\in \PB_{1,2;\epsilon}(n) \Leftrightarrow \pi \in \P_{1,2;\epsilon}(n)$. Let  $\PB_{2;\epsilon}(n):= \PB_2(n)\cap\PB_{1,2; \epsilon}(n)$.  

We introduce some partition statistics. 
Let $\Pair(\pi)$ be the set of pair blocks of a set partition $\pi$ and $\Sing(\pi)$ be the set of singletons of a set partition $\pi$. Let $\NB(\pi,f)$ be the set of negative blocks (i.e.\ blocks colored by $-1$) of a set partition $(\pi,f)$ of type B. For a set partition $\pi$ let $\Cr(\pi)$ be the number of crossings of $\pi$, i.e.\ 
$$
\Cr(\pi)=\#\{\{V,W\} \subset \pi \mid \text{There exist $i,j \in V, k,l\in W$ such that $i<k<j<l$}\}. 
$$
For two blocks $V, W$ of a set partition, we say that $W$ \emph{covers} $V$ if there are $i,j \in W$ such that $i <k <j$ for any $k\in V$. For $(\pi,f)\in\PB_{1,2}(n)$, let $\InS(\pi)$ be the number of pairs of a singleton and a covering block:   
$$
\InS(\pi)=\#\{(V,W) \in \pi \times \pi \mid V \text{~is a singleton,~} W \text{~covers~}V  \}. 
$$
Let $\InNB(\pi,f)$ be the number of pairs of a negative block and a covering block: 
$$
\InNB(\pi,f)=\#\{(V,W) \in \pi \times \pi \mid \#V=2, f(V)=-1, W \text{~covers~}V  \}. 
$$
Let $\SLNB(\pi,f)$ be the number of pairs of a negative block and a \emph{singleton to the left}:  
\begin{align*}
\SLNB(\pi,f)&=\#\{(V,W) \in \pi \times \pi \mid \#W=2, f(W)=-1, \\ 
&\quad\qquad V \text{~is a singleton $\{i\}$}, i < j \text{~for all~} j\in W \}. 
\end{align*}
For a block $V$ of a set partition $\pi$, let $\Cov(V)$ be the number of blocks of $\pi$ which cover  $V$ and let $\SL(V)$ be the number of singletons which are placed to the left of $V$. More precisely $\Cov(V),\SL(V)$ should be denoted as $\Cov(V;\pi), \SL(V;\pi)$, but we use the simpler notations. 

A set partition $\pi$ is \emph{noncrossing} if $\Cr(\pi)=0$. The set of noncrossing partitions of $[n]$ is denoted by $\NC(n)$ and the set of noncrossing pair partitions is denoted by $\NC_2(n)$ when $n$ is even. 
A block of a noncrossing partition is \emph{inner} if it is covered by another block. A block of a noncrossing partition is \emph{outer} if it is not inner. 
For $\pi\in \NC(n)$, let $\In(\pi)$ be the number of inner blocks of $\pi$ and let $\Out(\pi)$ be the number of outer blocks. Note that $\In(\pi)+\Out(\pi)$ equals the number of blocks of $\pi$.

\begin{theorem}\label{lem101}
For any $x_1,\dots,x_n \in H_\R$ and any $\epsilon=(\epsilon(1), \dots, \epsilon(n))\in\{1,\ast\}^n$, we have
\begin{equation}\label{formula101}
\begin{split}
&\B^{\epsilon(n)}(x_{n})\cdots \B^{\epsilon(1)}(x_1)\Omega \\
&\qquad= \sum_{(\pi,f)\in\PB_{1,2;\epsilon}(n)}  \alpha^{\NB(\pi,f)}q^{\Cr(\pi)+ \InS(\pi)+2 \InNB(\pi,f) + 2\SLNB(\pi,f)} \\
&\qquad\quad \times \prod_{\substack{\{i,j\} \in \Pair(\pi)\\ f(\{i,j\})=1} }\langle x_i, x_j\rangle \prod_{\substack{\{i,j\} \in \Pair(\pi)\\ f(\{i,j\})=-1}} \langle x_i, \overline{x_j}\rangle
\bigotimes_{i\in \Sing(\pi)} x_i, 
\end{split}
\end{equation}
where $\bigotimes_{i\in V} x_i$ denotes the tensor product $x_{v_1}\otimes \cdots \otimes x_{v_m}$ when $V$ is written as $V = \{v_1,\dots,v_m\} \subset \N$, $v_1<\cdots < v_m.$ If $V=\emptyset$, then we understand that  $\bigotimes_{i\in V} x_i=\Omega.$ Moreover, we may write the formula in terms of set partitions: 
\begin{align}\label{formula102}
&\B^{\epsilon(n)}(x_n)\cdots \B^{\epsilon(1)}(x_1)\Omega \notag\\
&\quad=\sum_{\pi\in \P_{1,2;\epsilon}(n)} q^{\Cr(\pi)+\InS(\pi)} \left(\prod_{\substack{\{i,j\} \in \Pair(\pi)} }\left(\langle x_i,x_j\rangle + \alpha q^{2 \Cov(\{i,j\}) +2 \SL(\{i,j\})}\langle x_i,\overline{x_j}\rangle\right)\right) \bigotimes_{i\in \Sing(\pi)} x_i. 
\end{align}
\end{theorem}
\begin{remark}
If $\#\{i\in[j] \mid \epsilon(i)=1\}>\#\{i\in[j] \mid \epsilon(i)=\ast\}$ for some $j\in[n]$, then  we have $\B^{\epsilon(n)}(x_{n})\cdots \B^{\epsilon(1)}(x_1)\Omega=0$. This case is also covered by \eqref{formula101} if we understand the sum over the empty set is 0 since $\PB_{1,2;\epsilon}(n)=\emptyset$ in this case. 
\end{remark}
\begin{proof}
The proof is given by induction. When $n=1$,  $\B(x_1)\Omega=0$ and $\B^\ast(x_1)\Omega=x_1$ and hence the formula is true. Suppose that the formula is true for $n=k$. Then for any $\epsilon\in\{1,\ast\}^{k+1}$, we get  
\begin{align*}
&\B^{\epsilon(k+1)}(x_{k+1})\cdots \B^{\epsilon(1)}(x_1)\Omega \\
&\quad= \sum_{(\pi,f)\in\PB_{1,2;\epsilon|_{[k]}}(k)}  \alpha^{\NB(\pi,f)}q^{\Cr(\pi)+ \InS(\pi)+ 2 \InNB(\pi,f) + 2\SLNB(\pi,f)} \\
&\qquad \times \prod_{\substack{\{i,j\} \in \Pair(\pi)\\ f(\{i,j\})=1} }\langle x_i, x_j\rangle \prod_{\substack{\{i,j\} \in \Pair(\pi)\\ f(\{i,j\})=-1}} \langle x_i, \overline{x_j}\rangle \, 
\B^{\epsilon(k+1)}(x_{k+1})\left(\bigotimes_{i\in \Sing(\pi)} x_i\right). 
\end{align*}
The operator $\B^{\epsilon(k+1)}(x_{k+1})$ equals $\r^\ast(x_{k+1})$ if $\epsilon(k+1)=\ast$ and $\r_q(x_{k+1})+ \alpha  \ell_{q}(\overline{x_{k+1}})q^{N-1}$ if $\epsilon(k+1)=1$ from Theorem \ref{thm1}. We will show that the action of $\B^{\epsilon(k+1)}(x_{k+1})$ corresponds to the inductive pictorial description of set partitions of type B. 

We fix $(\pi,f)\in \PB_{1,2;\epsilon|_{[k]}}(k)$ and suppose that $\pi$ has singletons  $k_1<\cdots <k_p <i < m_1 <\cdots <m_r$ (we understand that $p=0$ ($r=0$, respectively) when there is no singleton to the left (right, respectively) of $i$), pair blocks $U_1,\dots, U_s$ with color $1$ to the right of $i$, pair blocks $V_1,\dots, V_t$ with color $-1$ to the right of $i$ and pair blocks $W_1,\dots, W_u$ which cover $i$. There may be pair blocks to the left of $i$, but they do not matter. Note that when there is no singleton, the arguments below can be modified easily. 

Case 1. If $\epsilon(k+1)=\ast$, then the operator $\r^\ast(x_{k+1})$ acts on the tensor product, putting $x_{k+1}$ on the right. This operation pictorially corresponds to adding the singleton $\{k+1\}$ (with color 1) to $(\pi,f)\in\PB_{1,2;\epsilon|_{[k]}}(k)$, to yield the new type B partition $(\tilde{\pi},\tilde{f})\in\PB_{1,2;\epsilon}(k+1)$. This map $(\pi,f)\mapsto (\tilde{\pi},\tilde{f})$ does not change the numbers $\NB, \Cr, \InNB, \SLNB$ or $\InS$, which is compatible with the fact that the action of $\r^\ast(x_{k+1})$ does not change the coefficient. 
Note that if $\epsilon(k+1)=\ast$, then any $(\sigma,g) \in \PB_{1,2;\epsilon|_{[k+1]}}(k+1)$ has the singleton $\{k+1\}$. Hence the formula \eqref{formula101} holds when $n=k+1$ and $\epsilon(k+1)=\ast$.

Case 2. If $\epsilon(k+1)=1$, then we have two cases. 

Case 2a. If $\r_q(x_{k+1})$ acts on the tensor product, then new $p+r+1$ terms appear by using \eqref{rq}. In the $i^{\rm th}$ term the inner product $\langle x_{k+1},x_i\rangle$ appears with coefficient $q^r$. Pictorially this corresponds to getting a set partition $(\tilde{\pi}, \tilde{f}) \in \PB_{1,2;\epsilon|_{[k+1]}}(k+1)$  by adding $k+1$ to $(\pi,f)$ and creating the pair $\{i, k+1\}$ with color $1$. This pair crosses the blocks $W_1, \dots, W_u$ and so increases the crossing number by $u$ but decreases the number of inner singletons by $u$ because originally $i$ was the inner singleton of $W_1,\dots, W_u$. Now the new inner singletons $\{m_1\}, \dots, \{m_r\}$ and new inner negative blocks $V_1,\dots, V_t$ appear. Because $i$ is not a singleton in $(\tilde{\pi}, \tilde{f})$, the number of singletons left to negative blocks decreases by $t$. Altogether we have: 
$\Cr(\tilde{\pi})=\Cr(\pi)+u$, $\InS(\tilde{\pi})=\InS(\pi)-u+r$, $\InNB(\tilde{\pi},\tilde{f})=\InNB(\pi)+t$, $\SLNB(\tilde{\pi}, \tilde{f})= \SLNB(\pi,f)-t$ and $\NB(\tilde{\pi},\tilde{f})=\NB(\pi,f)$. So the exponent of $q$ increases by $r$. This factor $q^r$ is exactly the factor appearing in \eqref{rq} when $\r_q(x_{k+1})$ acts on $x_i$.

Case 2b. If $\alpha  \ell_{q}(\overline{x_{k+1}})q^{N-1}$ acts on the tensor product, then new $p+r+1$ terms appear by using \eqref{lq}. In the $i^{\rm th}$ term the inner product $\langle\overline{x_{k+1}},x_i\rangle$ appears with coefficient $\alpha q^{p+(p+r+1)-1}$ (note here that $p+r+1=\#\Sing(\pi)$). Pictorially this means that the new pair $\{i,k+1\}$ is created with color $-1$. Similarly to Case 2a, we count the change of numbers and get 
$\Cr(\tilde{\pi})=\Cr(\pi)+u$, $\InS(\tilde{\pi})=\InS(\pi)-u+r$, $\InNB(\tilde{\pi},\tilde{f})=\InNB(\pi)+t$, $\SLNB(\tilde{\pi}, \tilde{f})= \SLNB(\pi,f)-t+p$ and $\NB(\tilde{\pi},\tilde{f})=\NB(\pi,f)+1$. Altogether, when moving from $(\pi,f)$ to $(\tilde{\pi}, \tilde{f})$, the exponent of $\alpha$ increases by $1$ and the exponent of $q$ increases by $2p+r$, which coincides with the coefficient appearing in the action of $\alpha  \ell_{q}(\overline{x_{k+1}})q^{N-1}$, creating the inner product $\langle  \overline{x_{k+1}},x_i\rangle$. 

Note that as $(\pi,f)$ runs over $\PB_{1,2;(\epsilon(1),\dots,\epsilon(k))}(k)$, every set partition $(\tilde{\pi},\tilde{f}) \in \PB_{1,2;(\epsilon(1),\dots,\epsilon(k),1)}(k+1)$ appears exactly once in one of Case 2a and Case 2b. Therefore in Case 2, the pictorial inductive step and the actual action of $\B(x_{k+1})$ both create the same terms with the same coefficients, and hence the formula \eqref{formula101} is true when $n=k+1$ and $\epsilon(k+1)=1$. Case 1 and Case 2 show by induction that the formula \eqref{formula101} holds for all $n \in \N$.

We show the formula \eqref{formula102}. 
For $\pi =\{V_1, \dots, V_k, S_1, \dots, S_m\} \in \P_{1,2}(n)$ where $V_i=\{a_i,b_i\}$ are pairs and $S_j$ are singletons, we have 
\begin{align*}
&\prod_{i=1}^k \left(\langle x_{a_i},x_{b_i}\rangle + \alpha q^{2 (\Cov(V_i) + \SL(V_i)  )}\langle x_{a_i},\overline{x_{b_i}}\rangle\right)\\
 &= 
\sum_{(n_1,\dots,n_k)\in\{1,-1\}^k}\prod_{i=1}^k \langle x_{a_i},x_{b_i}\rangle ^{\delta_{n_i,1}} \left(\alpha q^{2 (\Cov(V_i) + \SL(V_i)  )}\langle x_{a_i},\overline{x_{b_i}}\rangle\right)^{\delta_{n_i,-1}} \\
&= 
\sum_{(n_1,\dots,n_k)\in\{1,-1\}^k}\left(\alpha^{\#\{i \in[k] \mid n_i=-1\}}q^{2\sum_{i=1}^k (\Cov(V_i) + \SL(V_i)  )\delta_{n_i,-1}}\prod_{i\in[k],n_i=1} \langle x_{a_i},x_{b_{i}}\rangle\prod_{i\in[k], n_i=-1} \langle x_{a_i},\overline{x_{b_i}}\rangle\right) \\ 
&= 
\sum_{f}\alpha^{\NB(\pi,f)}q^{2\InNB(\pi,f)+2\SLNB(\pi,f)}\prod_{\{i,j\}\in\pi,f(\{i,j\})=1} \langle x_i,x_j\rangle\prod_{\{i,j\}\in\pi,f(\{i,j\})=-1} \langle x_{i},\overline{x_j}\rangle, 
\end{align*}
where $f$ is defined by $f(V_i)=n_i$ and $f(S_j)=1$ and hence $f$ runs over the set of all colorings of $\pi$. Note here that $\NB(\pi,f)=\#\{i\in[k]\mid n_i=-1\}$, $\InNB(\pi,f)=\sum_{i=1}^k \Cov(V_i)\delta_{n_i,-1}$ and $\SLNB(\pi,f)=\sum_{i=1}^k \SL(V_i)\delta_{n_i,-1}$. 
\end{proof}
\begin{corollary}\label{thm2} 
Suppose that $x_1,\dots,x_n \in H_\R$ and $\epsilon\in\{1,\ast\}^n$. 
\begin{enumerate}[\rm(1)] 
\item 
\begin{align*}
&\langle\Omega, \B^{\epsilon(n)}(x_n)\cdots \B^{\epsilon(1)}(x_1)\Omega\rangle_{\alpha,q}\\
&\quad=
\sum_{(\pi,f) \in \PB_{2;\epsilon}(n)} \alpha^{\NB(\pi,f)}q^{\Cr(\pi)+2 \InNB(\pi,f)}\prod_{\substack{\{i,j\} \in \pi \\ f(\{i,j\})=1} }\langle x_i, x_j\rangle \prod_{\substack{\{i,j\} \in \pi\\ f(\{i,j\})=-1}} \langle x_i, \overline{x_j}\rangle \\
&\quad =
\sum_{\pi\in \P_{2;\epsilon}(n)} q^{\Cr(\pi)} \prod_{\substack{\{i,j\} \in\pi} }\left(\langle x_i,x_j\rangle + \alpha q^{2\Cov(\{i,j\})}\langle x_i,\overline{x_j}\rangle\right). 
\end{align*}
\item 
\begin{align*}
&\langle\Omega, \G(x_n)\cdots \G(x_1)\Omega\rangle_{\alpha,q}=\langle\Omega, \G(x_1)\cdots \G(x_n)\Omega\rangle_{\alpha,q}\\
&\quad=
\displaystyle\sum_{(\pi,f) \in \PB_2(n)} \alpha^{\NB(\pi,f)}q^{\Cr(\pi)+2 \InNB(\pi,f)}\prod_{\substack{\{i,j\} \in \pi \\ f(\{i,j\})=1} }\langle x_i, x_j\rangle \prod_{\substack{\{i,j\} \in \pi\\ f(\{i,j\})=-1}} \langle x_i,\overline{x_j}\rangle \\
&\quad=
\sum_{\pi\in \P_{2}(n)} q^{\Cr(\pi)} \prod_{\substack{\{i,j\} \in \pi} }\left(\langle x_i,x_j\rangle + \alpha q^{2 \Cov(\{i,j\})}\langle x_i,\overline{x_j}\rangle\right). 
\end{align*}
\end{enumerate}
Note that the sum over the empty set is understood to be 0. 
\end{corollary}
\begin{proof}
(1) is clear and (2) follows from (1) by taking the sum over all $\epsilon$. 
\end{proof}

\begin{corollary}\label{cor13} Assume that $x_i \in H_\R$ and $\overline{x_i}=x_i$ for $i=1,\dots,2m$. 
\begin{enumerate}[\rm(1)]
\item For $\alpha=0$, we recover the $q$-deformed formula for moments \cite[Proposition 2]{BS91}: 
\begin{align*}
\langle\Omega, G_{0,q}(x_1)\cdots G_{0,q}(x_{2 m})\Omega\rangle_{0,q}
=\sum_{\pi \in \P_2(2 m)} q^{\Cr(\pi)}\prod_{\{i,j\} \in\pi }\langle x_i, x_j\rangle. 
\end{align*}
\item For $q=0$, let $t=1/(1+\alpha)$ (in this case $P_{\alpha,0}^{(n)}=1+\alpha \pi_0$ is positive definite for $\alpha \in (-1, \infty)$ and so we may consider $t\in(0,\infty)$) and let $\tilde{G}_t(x)=\sqrt{t}G_{\alpha,0}(x)$. We recover the $t$-deformed formula for moments \cite[Theorem 7.3]{BW01}: 
\begin{align*}
\langle\Omega, \tilde{G}_t(x_1)\cdots \tilde{G}_t(x_{2 m}), \Omega\rangle_{(1-t)/t,0}
=\sum_{\pi \in \NC_2(2 m)} t^{\In(\pi)}\prod_{\{i,j\} \in\pi }\langle x_i, x_j\rangle. 
\end{align*}
\end{enumerate}
\end{corollary}
\begin{proof}
(1) is clear. 

(2)\,\, 
When $q=0$, the non zero contributions in Corollary \ref{thm2}(2) come only when $\pi$ is noncrossing and inner blocks are colored by 1. Given $V\in \pi \in \NC_2(2 m)$, $1+\alpha 0^{2\Cov(V)}$ equals $1+\alpha$ if $V$ is outer and 1 otherwise. Hence
\begin{align*}
\langle\Omega, G_{\alpha,0}(x_1)\cdots G_{\alpha,0}(x_{2 m})\Omega\rangle_{\alpha,0}
&=\sum_{\pi \in \NC_2(2 m)} \left(1+\alpha\right)^{\Out(\pi)}\prod_{\{i,j\} \in\pi }\langle x_i, x_j
\rangle \\
&=\sum_{\pi \in \NC_2(2 m)} t^{-\Out(\pi)}\prod_{\{i,j\} \in\pi }\langle x_i, x_j\rangle.  
\end{align*}
The conclusion follows by multiplying this formula by $t^m$ and by using the relation $m-\Out(\pi)=\In(\pi)$. 
\end{proof}
\begin{remark} 
It seems that the $t$-deformation of the classical Gaussian operator proposed in \cite[Sections 8,9]{BW01} is not contained in our family $\G(x)$. 
\end{remark}

\subsection{Traciality of the vacuum state} 
Let $\A$ be the von Neumann algebra generated by $\{\G(x)\mid x\in H_\R\}$ acting on the completion of $\F$ regarding the inner product $\langle\cdot,\cdot\rangle_{\alpha,q}$. 
\begin{proposition} Let $\alpha,q \in (-1,1)$.  Suppose that $\text{\rm dim}(H_\R) \geq2$. Then the vacuum state is a trace on $\A$  if and only if $\alpha=0$. 
\end{proposition}
\begin{proof} By using Corollary \ref{thm2}, we obtain 
\begin{align*}
&\langle\Omega, \G(x_1)\G(x_2)\G(x_3)\G(x_4)\Omega\rangle_{\alpha,q} \\
&\qquad= 
\langle x_1, x_2\rangle\langle x_3, x_4\rangle 
+\alpha\langle x_1, \overline{x_2}\rangle\langle x_3, x_4\rangle 
+\alpha\langle x_1, x_2\rangle\langle x_3, \overline{x_4}\rangle 
+\alpha^2\langle x_1, \overline{x_2}\rangle\langle x_3, \overline{x_4}\rangle \\
&\qquad \quad
+ q \langle x_1, x_3\rangle\langle x_2, x_4\rangle 
+\alpha q \langle x_1, \overline{x_3}\rangle\langle x_2, x_4\rangle 
+ \alpha q \langle x_1, x_3\rangle\langle x_2, \overline{x_4}\rangle
+\alpha^2 q \langle x_1, \overline{x_3}\rangle\langle x_2, \overline{x_4}\rangle \\
&\qquad\quad
+\langle x_1, x_4\rangle\langle x_2, x_3\rangle
+\alpha \langle x_1, \overline{x_4}\rangle\langle x_2, x_3\rangle
+ \alpha q^2 \langle x_1, x_4\rangle\langle x_2, \overline{x_3}\rangle 
+\alpha^2 q^2\langle x_1, \overline{x_4}\rangle\langle x_2, \overline{x_3}\rangle 
\end{align*}
and by permuting $x_1,x_2,x_3,x_4$, 
\begin{align*}
&\langle\Omega, \G(x_2)\G(x_3)\G(x_4)\G(x_1)\Omega\rangle_{\alpha,q} \\
&\qquad= 
\langle x_2, x_3\rangle\langle x_1, x_4\rangle 
+\alpha\langle x_2, \overline{x_3}\rangle\langle x_1, x_4\rangle 
+\alpha\langle x_2, x_3\rangle\langle x_1, \overline{x_4}\rangle 
+\alpha^2\langle x_2, \overline{x_3}\rangle\langle x_1, \overline{x_4}\rangle \\
&\qquad \quad
+ q \langle x_2, x_4\rangle\langle x_1, x_3\rangle 
+\alpha q \langle x_2, \overline{x_4}\rangle\langle x_1, x_3\rangle 
+ \alpha q \langle x_2, x_4\rangle\langle x_1, \overline{x_3}\rangle
+\alpha^2 q \langle x_2, \overline{x_4}\rangle\langle x_1, \overline{x_3}\rangle \\
&\qquad\quad
+\langle x_1, x_2\rangle\langle x_3, x_4\rangle
+\alpha \langle x_1, \overline{x_2}\rangle\langle x_3, x_4\rangle
+ \alpha q^2 \langle x_1, x_2\rangle\langle x_3, \overline{x_4}\rangle 
+\alpha^2 q^2\langle x_1, \overline{x_2}\rangle\langle x_3, \overline{x_4}\rangle.  
\end{align*}
Since \text{dim}$(H_\R) \geq 2$, there are two orthogonal unit eigenvectors $e_1,e_2$ of the involution ${}^-$, and we take $x_1=x_3=e_1$ and $x_2=x_4=e_2$. Note that for some $s,t \in\{\pm1\}$ we have $\overline{e_1}= s e_1, \overline{e_3}= t e_3$. Hence 
\begin{equation}\label{eq331}
\begin{split}
&\langle\Omega, \G(x_1)\G(x_2)\G(x_3)\G(x_4)\Omega\rangle_{\alpha,q}\\
&\qquad-\langle\Omega, \G(x_2)\G(x_3)\G(x_4)\G(x_1)\Omega\rangle_{\alpha,q} \\
&\qquad\qquad= 4 t\alpha (1-q^2)(1+s \alpha). 
\end{split}
\end{equation}
Therefore the vacuum state is not a trace when $\alpha \neq 0$. When $\alpha=0$, the von Neumann algebra becomes the $q$-deformed  von Neumann algebra and the traciality is known in \cite[Theorem 4.4]{BS94}. 
\end{proof}

\begin{center} Open Problems and Remarks
\end{center}
\begin{itemize}
\item It would be interesting to ask if the von Neumann algebra $\A$ is factorial. In the $\alpha=0$ case the factoriality was proved by Bo\.zejko, K\"ummerer and Speicher \cite{BKS97} and by {\'S}niady in special cases and then by Ricard \cite{R05} in full generality. Moreover Nou showed that $\text{vN}(G_{0,q}(H_\R))$ is not injective if $\text{dim}(H_\R) \geq2$~\cite{N04}. If $\alpha=q=0$ then Voiculescu \cite{V85} showed that $\text{vN}(G_{0,0}(H_\R))$ is isomorphic to the free group factor $L(\mathbb{F}_{\text{dim}(H_\R)})$. 
Recently Guionnet and Shlyakhtenko \cite{GS14} showed that for each integer $n\geq2$ there exists $q(n)>0$ such that $\text{vN}(G_{0,q}(H_\R))$ is isomorphic to the free group factor $L(\mathbb{F}_{\text{dim}(H_\R)})$ for $|q|<q({\text{dim}(H_\R)})$. Their method was extended by Nelson \cite{N} to the non tracial case and he showed that Hiai's $q$-deformed Araki-Woods algebras \cite{H03} are isomorphic to Shlyakhtenko's free Araki-Woods algebras \cite{S97} for small $|q|$. 

\item It would be worth to ask whether the von Neumanan algebra $\A$ has the completely bounded approximation property.

\item Determine all the parameters $(\alpha,q)$ for which the type B symmetrization operator $P_{\alpha,q}^{(n)}$ is positive. The positivity is known for $\alpha,q\in[-1,1]$ and also for $q=0$, $\alpha\in[-1,\infty)$ as explained in Corollary \ref{cor13}. Then determine all the parameters for which $\text{Ker}(P_{\alpha,q}^{(n)})=\{0\}$ for all $n \geq1$.

\end{itemize}

\begin{center} Acknowledgments
\end{center}

The work was partially supported by the MAESTRO 
grant DEC-2011/02/A/ ST1/00119 (M.\ Bo\.zejko) and OPUS grant DEC-2012/05/B/ST1/00626 of National
Centre of Science (M.\ Bo\.zejko and W.\ Ejsmont). The work was partially supported by Marie Curie International Incoming Fellowships (Project 328112 ICNCP, T.\ Hasebe). 
T.\ Hasebe thanks Nobuhiro Asai for his interests and useful comments.

\end{document}